\documentclass{amsart}
\usepackage{amsmath}
\usepackage{amsthm}
\usepackage{amssymb}
\usepackage{amsfonts}
\usepackage{latexsym}
\usepackage{mathrsfs}
\usepackage{eucal}
\usepackage{hyperref}

\numberwithin{equation}{section}
\numberwithin{subsection}{section}

\newtheorem{thm}[subsection]{Theorem}
\newtheorem*{thm*}{Theorem}

\newtheorem{lemma}[subsection]{Lemma}
\newtheorem{prop}[subsection]{Proposition}
\newtheorem{cor}[subsection]{Corollary}

\theoremstyle{definition}
\newtheorem{defn}[subsection]{Definition}
\newtheorem{remark}[subsection]{Remark}

\usepackage{color}

\newcommand{\multH}{\mr{Mult} \, \cH}

\newcommand{\be}{\begin{equation}}
\newcommand{\ee}{\end{equation}}
\newcommand{\ba}{\begin{eqnarray}}
\newcommand{\ea}{\end{eqnarray}}
\newcommand{\bal}{\begin{align*}}
\newcommand{\eal}{\end{align*}}
\newcommand{\baln}{\begin{align}}
\newcommand{\ealn}{\end{align}}
\newcommand{\bi}{\begin{itemize}}
\newcommand{\ei}{\end{itemize}}
\newcommand{\bn}{\begin{enumerate}}
\newcommand{\en}{\end{enumerate}}
\newcommand{\bbm}{\begin{bmatrix}}
\newcommand{\ebm}{\end{bmatrix}}
\newcommand{\bpm}{\begin{pmatrix}}
\newcommand{\epm}{\end{pmatrix}}
\newcommand{\bsm}{\left ( \begin{smallmatrix}}
\newcommand{\esm}{\end{smallmatrix} \right) }
\newcommand{\bp}{\begin{proof}}
\newcommand{\ep}{\end{proof}}
\newcommand{\nn}{\nonumber}

\newcommand{\mr}{\ensuremath{\mathrm}}

\newcommand{\mc}{\ensuremath{\mathcal}}

\renewcommand{\a}{\ensuremath{\alpha }}

\def\C{\mathbb{C}}

\def\nbdom{\mathrm{Dom} \, }
\def\nbran{\mathrm{Ran} \, }
\def\nbker{\mathrm{Ker} \, }

\newcommand{\ov}{\ensuremath{\overline}}

\newcommand{\cH}{\ensuremath{\mathcal{H}}}
\newcommand{\cJ}{\ensuremath{\mathcal{J} }}
\def\cK{\mathcal{K}}

\newcommand{\dom}[1]{\ensuremath{\mathrm{Dom} ({#1}) }}
\newcommand{\ran}[1]{\ensuremath{\mathrm{Ran} \left( {#1} \right) }}

\title{Unbounded multipliers of complete Pick spaces}

\author{Michael T. Jury}
\address{University of Florida}
\email{mjury@ufl.edu}

\author{Robert T.W. Martin}
\address{University of Manitoba}
\email{rtwmartin@gmail.com}

\begin{document}
\date{\today}
\begin{abstract}
We examine densely defined (but possibly unbounded) multiplication operators in Hilbert function spaces possessing a complete Nevanlinna-Pick (CNP) kernel. For such a densely defined operator $T$, the domains of $T$ and $T^*$ are reproducing kernel Hilbert spaces contractively contained in the ambient space. We study several aspects of these spaces, especially the domain of $T^*$, which can be viewed as analogs of the classical deBranges-Rovnyak spaces in the unit disk. 
  \end{abstract}
\thanks{First named author supported by NSF grant DMS-1900364}
\thanks{Second named author supported by NSERC grant 2020-05683}
\maketitle

\section{Introduction}\label{sec:intro}

A {\em Hilbert function space}, or \emph{reproducing kernel Hilbert space (RKHS)} is a Hilbert space $\mathcal H$ consisting of functions $f:X\to \mathbb C$ on some set $X$, with the property that for each $x\in X$, the point evaluation $f\to f(x)$ is bounded for the norm in $\mathcal H$. 

For any Hilbert function space $\mathcal H$ over a set $X$, we may consider its {\em multipliers}, which are the functions $\varphi:X\to \mathbb C$ for which $\varphi f\in \mathcal H$ whenever $f\in \mathcal H$. By the closed graph theorem, if $\varphi$ is a multiplier of $\mathcal H$ then the operator
\[
M_\varphi:f\to \varphi f
\]
is bounded on $\mathcal H$, and it is then straightforward to show that the function $\varphi$ is necessarily bounded on $X$, with the estimate
\[
|\varphi(x)|\leq \|M_\varphi \| \quad \text{ for all } x\in X.
\]
A basic problem in the theory is to characterize the set of multipliers $\mr{Mult} \, \cH$ for a given Hilbert function space $\mathcal H$. For example, when $\mathcal H$ is the classical Hardy space $H^2$ in the unit disk, it is well known that $\mr{Mult} \, H^2 $ coincides with the algebra $H^\infty$ of all bounded functions in the unit disk.

In this paper we are concerned with objects we will call {\em unbounded multipliers}, by which we mean functions $h:X\to \mathbb C$ for which we do not assume $hf$ belongs to $\mathcal H$ for all $f\in \mathcal H$, but only for $f$ lying in some dense subspace of $\mathcal H$.

We have two primary motivations: the first is a theorem of Su\'arez and Sarason \cite{sarason-2008}, which says that for the classical Hardy space $H^2$ in the unit disk $|z|<1$, a function $h$ gives a densely defined multiplier of $H^2$ if and only if it belongs to the {\em Smirnov class} $N^+$. This is the set of all analytic functions $h$ in the disk which can be represented in the form
\[
h= \frac{b}{a}
\]
where $a,b$ are bounded analytic functions in the disk, and additionally $a$ is a so-called {\em outer function}, which means that $a\cdot H^\infty$ is dense in $H^2$. Thus, given any Hilbert function space $\mathcal H$, one may ask for a description of the unbounded (densely defined) multipliers.

The second motivation comes from the theory of deBranges-Rovnyak spaces, and another theorem of Sarason which connects these spaces to unbounded multipliers of $H^2$. First of all, it is proved in \cite[Proposition 3.1]{sarason-2008} that if $h=b/a$ is a function in the Smirnov class, the functions $a,b$ can be chosen to be bounded by $1$ in the disk, and in such a way that
\begin{equation}\label{eqn:pythagorean-mate}
  |a(z)|^2+|b(z)|^2 =1 \quad \text{for a.e. } |z|=1.
\end{equation}
The $a,b$ with these additional properties are unique up to an overall unimodular constant factor. To the functions $a$ and $b$ one can associate two Hilbert function spaces in the unit disk, denoted $\mathcal M(a)$ and $\mathcal H(b)$, which are the spaces with reproducing kernels
\begin{equation}\label{eqn:ma-hb-kernels}
  \frac{a(z)a(w)^*}{1-zw^*}, \quad \frac{1-b(z)b(w)^*}{1-zw^*}
\end{equation}
respectively. Both of these spaces are contractively contained in the Hardy space $H^2$.  The space $\mathcal H(b)$ is called the {\em deBranges-Rovnyak space} associated to $b$, and the theory of such spaces takes on a rather different character if $b$ is either an extreme point of the unit ball of $H^\infty$, or non-extreme. The condition (\ref{eqn:pythagorean-mate}) implies that, in the case we are discussing, the function $b$ is not an extreme point of the unit ball of $H^\infty$. The important observation is that, when $b$ is non-extreme, there is always a unique (up to a unimodular constant) outer function $a$ so that (\ref{eqn:pythagorean-mate}) holds. Returning now to the Smirnov function $h=b/a$, what is proved in \cite[Propositions 5.3 and 5.4]{sarason-2008} is that the domain of the unbounded multiplication operator $T$ is precisely $\mathcal M(a)$, and the domain of $T^*$ is precisely $\mathcal H(b)$, and moreover these identifications are isometric if the domains of the (closed) unbounded operators are equipped with the graph norm. (For the result about $T^*$, see also \cite{Suarez}.) In particular, every $\mathcal H(b)$ space associated to a non-extreme $b$ arises as the domain of the adjoint of an unbounded multiplier. It is remarkable that many of the standard properties of the non-extreme $\mathcal H(b)$ spaces can be proved by adopting this point of view.

Now, since spaces with complete Pick kernel (which we review in the following subsection) are in many ways analogous to $H^2$, it is natural to wonder how much (if any) of this theory can be carried over. Analogues of the $\mathcal H(b)$ in complete Pick spaces have been considered in \cite{jury-2014}, where an analogue of the extreme/non-extreme distinction is observed, though at that time we were able to prove little about the non-extreme case. The main theme of the present paper is that if we re-interpret the non-extreme $\mathcal H(b)$, considering them as defined not by the form of the kernel as in (\ref{eqn:ma-hb-kernels}) but as the domains of $T^*$, for $h$ an unbounded multiplier, then many of the features of the theory of the non-extreme $\mathcal H(b)$ spaces do in fact carry over to the setting of complete Pick kernels. Thus, our program is to investigate unbounded multipliers of CNP spaces, and their adjoints, and in particular to study properties of their domains, viewed as Hilbert function spaces analogous to $\mathcal M(a)$ and $\mathcal H(b)$.

\subsection{Spaces with complete Pick kernel}  By a {\em CNP} kernel on a set $X$ we mean a reproducing kernel of the form
\[
  k(x,y) =\frac{1}{1-\langle u(x), u(y)\rangle}
\]
where $u:X\to \ell^2(I)$ is an injective function from $X$ into the Hilbert space $\ell^2(I)$ (here $I$ is some index set), for which $\|u(x)\|<1$ for all $x\in X$. Letting $E$ denote the range of $u$, we see that the Hilbert function space $\mathcal H(k)$ on $X$ is identified with the closed subspace of $H^2_I$ spanned by the restrictions of Drury-Arveson kernel functions $k(z,w)=\frac{1}{1-\langle z,w\rangle}$ to the set $E$. By this mechanism, CNP spaces inherit many properties of the $H^2_I$ spaces. In particular we will make use of three key facts about CNP spaces: first, a general form of the Nevanlinna-Pick interpolation theorem holds \cite{ball-trent-vinnikov-2001}, second, there is a Beurling-type theorem characterizing the closed, $\mr{Mult} \, \cH$-invariant subspaces of $\cH$ \cite[]{mccullough-trent-2000}, and third, every $f\in \cH$ belongs to the {\em Smirnov class $N^+(\cH)$}; this means that for every $f\in \cH$ there exist $\varphi, \psi\in \multH$ so that $f=\frac{\varphi}{\psi}$. For this there are two different proofs, see either \cite[Theorem 1.1]{AHMR-2017} or \cite[Corollary 3.4]{jury-martin-2021}. Finally, we note that by construction, our CNP spaces $\cH$ will always contain the constant functions, so that $\multH\subset\cH$.

\subsection{Readers' guide} In Section~\ref{sec:affiliated} we prove the basic characterization theorem (Theorem~\ref{thm:affiliated-implies-multiplication-CNP}), which says that a densely defined operator in a CNP space, which commutes with all bounded multipliers, must have the form $f\to hf$ for some function $h$; moreover we show that there exist bounded multipliers $a,b$ so that $b=ah$. In Section~\ref{sec:kernels} we study the domains of an unbounded multiplier $T$ and its adjoint $T^*$ as reproducing kernel Hilbert space contractively contained in $\mathcal H$. In particular we show, as in the case of the classical deBranges-Rovnyak spaces, the reproducing kernels $k_x$ for $\mathcal H$ are contained in $\nbdom T^*$ and their span is dense in $\nbdom T^*$. In Section~\ref{sec:representing-pairs} we examine the functions $a,b$ satisfying $b=ah$ more closely, giving a parameterization of all such pairs and we obtain some containment relations between various reproducing kernel Hilbert spaces associated to them. Sections~\ref{sec:corona} and \ref{sec:invariant-subspaces} explore further properties of $\nbdom T^*$, and we generalize two theorems of Sarason about deBranges-Rovnyak spaces to the present setting. In Section~\ref{sec:constructive} we give another, more constructive proof that the kernels are dense in $\nbdom T^*$, which can also be used to show that the polynomials are dense in $\nbdom T^*$ in certain cases.  Finally in Section~\ref{sec:questions} we make some final remarks and pose some questions. 

\section{Operators affiliated to $\mr{Mult} \, \cH$}\label{sec:affiliated}

\begin{prop}\label{prop:multipliers-are-closeable-CNP}
  Let $\mathcal H$ be a reproducing kernel Hilbert space on a set $X$. If $h:X\to \mathbb C$ is a function such that the multiplication operator
  \[
    f\to hf
  \]
  is densely defined in $\mathcal H$, then $h$ is closeable, in particular $f\to hf$ is closed when equipped with its maximal domain
  \[
    \nbdom h =\{f\in \mathcal H: hf\in\mathcal H\}
  \]
\end{prop}
\begin{proof}
  The proof is straightforward:  if $(f_n)$ is a sequence in $\nbdom h$  such that $f_n\to 0$ and $hf_n\to g$, then since point evaluations are bounded, we have that $f_n(x)\to 0$ for all $x\in X$, so $g(x) =\lim h(x)f_n(x)=0$.
\end{proof}
Observe that for a densely defined multiplication operator $f\to hf$ as in the last proposition, it is immediate that its maximal domain $\nbdom h$ is invariant for all bounded multiplication operators on $\mathcal H$. Theorem~\ref{thm:affiliated-implies-multiplication-CNP} below is a converse to this in the case that $\mathcal H$ is a CNP space (together with a stronger statement about the structure of $h$). To state it, we need one definition:
\begin{defn} Let $\mathcal H$ be a Hilbert function space. Say that a densely defined operator $T:\nbdom T \to \mathcal H$ is {\em affiliated to $\mr{Mult} \, \cH$} if:
  \begin{itemize}
  \item  whenever $f\in \nbdom T $ and $\varphi\in \mr{Mult} \, \cH$, then $\varphi f\in \nbdom T $, and $T(\varphi f) = \varphi\cdot (Tf)$.
  \end{itemize}
\end{defn}
\begin{prop}\label{prop:affiliated-implies-closeable} If $\mathcal H$ is a CNP space and $T$ is a densely defined operator in $\mathcal H$, affiliated to $\mr{Mult} \, \cH$, then $T$ is closeable.
  \end{prop}
In the language of \cite{bercetal-2010}, we are saying that the operator algebra $\mr{Mult} \, \cH$ has the {\em closability property} (\cite[Definition 3.2]{bercetal-2010}), whenever $\mathcal H$ is a CNP space. 
\begin{proof}
We recall the following definition \cite[Definition 4.1]{bercetal-2010} (with a slight change of notation): let $\mathcal A\subset \mathcal B(H)$ be a unital operator algebra. We say that a vector $f_0\in \mathcal H$ is a {\em rationally strictly cyclic vector} for $\mathcal A$ if for every vector $f\in \mathcal H$, there exist $a,b\in\mathcal A$ so that $af=bf_0$. Now if $\mathcal H$ is a CNP space, then, as we noted in the introduction, every $f\in \mathcal H$ can be expressed in the form $f=b/a$ where $a,b\in \mr{Mult}(\mathcal H)$, which says that the constant function $1\in\mathcal H$ is a rationally strictly cyclic vector for the unital operator algebra $\mr{Mult} \, \cH$. It is then immediate from \cite[Theorem 4.4]{bercetal-2010} that $\mr{Mult} \, \cH$ has the closability property. 
  \end{proof}

We are now ready for the main theorem of this section:

\begin{thm}\label{thm:affiliated-implies-multiplication-CNP}
  Let $\mathcal H$ be a CNP space on a set $X$, with kernel $k$. If $T$ is a densely defined operator in $\mathcal H$, affiliated to $\mr{Mult} \, \cH$, then there exists a function $h:X\to \mathbb C$ such that for every $f\in \nbdom T $, 
\begin{equation}
(Tf)(x)=h(x)f(x). 
\end{equation}
Moreover there exist nonzero multipliers $a,b\in \mr{Mult} \, \cH$ such that $b(x)=a(x)h(x)$ for all $x\in X$.
\end{thm}
\begin{proof}
  By Proposition~\ref{prop:affiliated-implies-closeable}, the operator $T$ is closeable, so it suffices to assume that $T$ is closed.  Next, we note that by the theorem of Agler and McCarthy \cite{agler-mccarthy-2001} we may assume that $X$ is a subset of the open unit ball of a Hilbert space $\ell^2(I)$, and the kernel $k$ is the restriction of the Drury-Arveson kernel
  \[
    k(z,w) =\frac{1}{1-\langle z,w\rangle}
  \]
  to this set. In particular $X$ becomes a metric space as a subset of $\ell^2(I)$, and every function $f\in \mathcal H$ (and every multiplier $\varphi\in \mr{Mult} \, \cH$) is continuous on $X$. 

  Now we begin the proof. Since $T$ is closed, its graph
\begin{equation}
G(T):= \left\{ \begin{pmatrix} f \\ Tf\end{pmatrix} : f\in \nbdom T \right\}
\end{equation}
is a closed subspace of $\mathcal H\oplus \mathcal H$. To say that $T$ is affiliated to $\mr{Mult} \, \cH$ means that this subspace is invariant under the action of $M_\varphi\oplus M_\varphi$ for every bounded multiplier $\varphi\in \mr{Mult} \, \cH$. It follows from the McCullough-Trent Beurling theorem for CNP spaces \cite{mccullough-trent-2000} that there exists an auxiliary Hilbert space $\mathcal E$ and a partially isometric multiplier $\Theta=\begin{pmatrix} A \\ B\end{pmatrix}$ from $\mathcal H\otimes \mathcal E$ to $\mathcal H\otimes \mathbb C^2$ such that $G(T)$ is the range of $M_\Theta$. Thus, we can re-express the graph of $T$ as
\begin{equation}\label{eqn:beurling-graph-CNP}
G(T):= \left\{ \begin{pmatrix} AF\\ BF\end{pmatrix} : F\in \mathcal H\otimes \mathcal E\right\}
\end{equation}
This says first of all that $\nbdom T $ is equal to the range of the row multiplier $M_A:\mathcal H\otimes \mathcal E\to \mathcal H$. 

Now, fix a basis $\{e_n\}$ of $\mathcal E$, with respect to this basis the multipliers $A$ and $B$ can be written as rows $A=(a_1 \ a_2\ \dots), B=(b_1, b_2, \dots)$, with each $a_k$ and $b_k$ a bounded (contractive) multiplier, with $a_k$ belonging to $\nbdom T $, and $b_k=Ta_k$ for each $k$. 

We now construct the function $h$. Since $\nbdom T $ is assumed dense in $\mathcal H$, the range of $A$ is dense, and so in particular for each $x\in X$ there exists an index $k$ for which $a_k(x)\neq 0$. Fix a point $x_0$, and such a $k$. Then there exists a neighborhood of $x_0$ in $X$ in which $a_k$ is nonvanishing, and we define $h(x) = b_k(x)/a_k(x)$ in this neighborhood. To see that $h$ is well-defined, suppose $j$ is another index with $a_j(x_0)\neq 0$. Consider the $2\times 2$ matrix multiplier
\begin{equation}
C = \begin{pmatrix} a_j & a_k\\ b_j & b_k \end{pmatrix}.
\end{equation}
By construction, the range of the multiplier $M_C$ is contained in the range of the multiplier $M_\Theta$, and hence in the graph of $T$.  If we apply $C$ to the column vector $\begin{pmatrix} -a_k \\ a_j\end{pmatrix}\in \mathcal H\oplus \mathcal H$, we get
\begin{equation}
\begin{pmatrix} a_j & a_k\\ b_j & b_k \end{pmatrix}\begin{pmatrix} -a_k \\ a_j\end{pmatrix} = \begin{pmatrix} 0 \\ a_jb_k - a_kb_j\end{pmatrix}.
\end{equation}
Since $T$ is a linear operator, we have $T(0)=0$ and hence $a_j(x)b_k(x) -a_k(x)b_j(x)=\det C(x)\equiv 0$.  It follows that the rows of $C(x)$ are linearly dependent for each $x$. By the choice of $a_j$ and $a_k$, at our fixed point $x_0$ we have $b_k(x_0) = a_k(x_0)h(x_0)$, so by the linear dependence of the rows we have also $b_j(x_0) = a_j(x_0)h(x_0)$, which shows that $h$ is well-defined.

Next we observe that for each $k$, if $a_k(x)=0$ for some $x$, then also $b_k(x)=0$. Indeed, fix a $x_0$ with $a_k(x_0)=0$ and choose an index $j$ so that $a_j(x_0)\neq 0$. We form the $2\times 2$ matrix $C$ as before, whose determinant is identically $0$. At our chosen $x_0$, we have $\det C(x_0)=a_j(x_0)b_k(x_0)=0$, which forces $b_k(x_0)=0$.

For the function $h(x)$ just constructed, we now have $b_k(x)=a_k(x)h(x)$ for every $k$ and all $x\in X$. It follows that for every $F\in \mathcal H\otimes \mathcal E$ we have $B(x)F(x)=h(x)A(x)F(x)$. But from (\ref{eqn:beurling-graph-CNP}), we see that $T(AF)=BF$, and hence for every $f\in \nbdom T $, we have $(Tf)(x)=h(x)f(x)$, which completes the proof. 
\end{proof}

\subsection{Examples}\label{subsec:examples}  A general way to construct densely defined multipliers is as follows: if $a\in Mult (\mathcal H)$ is any bounded multiplier that is also a cyclic vector for $\mathcal H$, that is, the set $a\mathcal H=\{af: f\in \mathcal H\}$ is dense in $\mathcal H$, then for any multiplier $b\in Mult (\mathcal H)$ the quotient
\[
  h:=\frac{b}{a}
\]
will be  a densely defined multiplier; indeed it is obvious that $\nbdom h$ contains $a\mathcal H$. (The set of all $h:X\to \mathbb C$ representable in this way is called the {\em Smirnov class} associated to $\mathcal H$, denoted $N^+(\mathcal H)$.) By a theorem of Sarason \cite{sarason-2008}, in the case of $\mathcal H=H^2(\mathbb D)$ it is known that, conversely, every densely defined multiplier belongs to the Smirnov class. We do not know if this remains true for general CNP spaces, though we can prove it in some special cases (see the end of this section, and Section~\ref{sec:questions} for further discussion).

A characterizing feature of CNP spaces is the following: let $\mathcal H$ be a CNP space over a set $X$, with kernel $k$. For any nonempty subset $Y\subset X$, we can restrict $k$ to $Y\times Y$ to obtain a CNP space $\mathcal H_Y$ on $Y$. If now $\varphi\in Mult (\mathcal H_Y)$, then $\varphi$ can be extended to a function $\psi$ on all of $X$, and in such a way that $\|\psi\|_{Mult \mathcal H} = \|\varphi\|_{Mult \mathcal H_Y}$. One can ask an analogous question for densely defined multipliers: if $h:Y\to \mathbb C$ is a densely defined multiplier of $\mathcal H_Y$, does $h$ extend to a function $g:X\to \mathbb C$ in such a way that $g$ is a densely defined multiplier of $\mathcal H$? The following example shows that the answer in general is no.

Let $X=\mathbb D$ be the unit disk in the complex plane; the Szeg\H{o} kernel $k(z,w)=\frac{1}{1-z\overline{w}}$ is a CNP kernel on $\mathbb D$, the associated Hilbert function space is of course the Hardy space $H^2$. Let $Y=\{y_n\}_{n=0}^\infty\subset \mathbb D$ be an interpolating sequence for $H^2$. This implies that every bounded function $\varphi :Y\to \mathbb C$ is a bounded multiplier of $\mathcal H_Y$, and thus extends to a bounded multiplier $g\in Mult (H^2)=H^\infty$.

\begin{prop} Let $Y\subset \mathbb D$ be an interpolating sequence. There exists a function $h:Y\to \mathbb C$ which is a densely defined multiplier of $\mathcal H_Y$, but which cannot be extended to a densely defined multiplier of $H^2$.
  \end{prop}
  \begin{proof}
    The proof consists in combining three observations: first, {\em every} function $h:Y\to \mathbb C$ is a densely defined multiplier of $\mathcal H_Y$. Indeed, since $Y$ is an interpolating sequence, the set $\mathcal F$ of finitely supported functions $f:Y\to \mathbb C$ is dense in $\mathcal H_Y$, and evidently $h\mathcal F\subset \mathcal F\subset \mathcal H_Y$, so that every $h$ has dense domain. Second, we use the result of Sarason mentioned above, that the densely defined multipliers of $H^2$ coincide with the Smirnov class $N^+$. Finally, if $g$ is any Smirnov class function, there exists a constant $C$ (depending on $g$) so that
    \begin{equation}\label{eqn:smirnov-growth-condition}
      |g(z)|\leq e^{\frac{C}{1-|z|}}
    \end{equation}
    for all $|z|<1$ \cite[II.3.1]{privalov-1950}. Thus, by selecting $h$ so that $|h(y_n)|$ goes to infinity so rapidly that (\ref{eqn:smirnov-growth-condition}) cannot hold for any $C$, (say $h(y_n) = \exp ((1-|y_n|)^{-2})$), we see that $h$ cannot extend to a Smirnov function in $\mathbb D$.  
  \end{proof}
  On the other hand, still supposing $Y$ is an interpolating sequence, every unbounded multiplier $h$ of $\mathcal H_Y$ (thus, every function $h:Y\to \mathbb C$) belongs to the ``local'' Smirnov class $N^+(\mathcal H_Y)$:
  \begin{prop} Let $Y\subset \mathbb D$ be an interpolating sequence. Then every function $h:Y\to \mathbb C$ belongs to $N^+(\mathcal H_Y)$.
  \end{prop}
  \begin{proof}
    Again, the proof is quite simple. Since $Y$ is an interpolating sequence, every bounded function $\varphi:Y\to \mathbb C$ is a bounded multiplier of $\mathcal H_Y$, and it is also easy to see that every nonvanishing function $f\in \mathcal H_Y$ is a cyclic vector, since if $f$ is nonvanishing the set of products $\{\varphi f:\varphi\in \mr{Mult} \, \cH\}$ will contain the set $\mathcal F$ of all finitely supported functions on $Y$, which is dense in $\mathcal H$. It remains only to note that any numerical sequence $(x_n)$ can be expressed as a ratio of two bounded sequences
    \[
      x_n=\frac{b_n}{a_n},
    \]
    with $0<a_n\leq 1$ for all $n$. Doing this for $x_n=h(y_n)$, the functions $a(y_n):=a_n$ and $b(y_n):=b_n$ are bounded multipliers of $\mathcal H_Y$, with $a$ cyclic since it is nonvanishing. 
    \end{proof}

\section{Reproducing kernels for $\nbdom T$ and $\nbdom T ^*$}\label{sec:kernels}
In this section we fix a CNP space $\mathcal H$ on a set $X$, and closed, densely defined operator $T$ in $\mathcal H$, affiliated to $\mr{Mult} \, \cH$.  By Theorem~\ref{thm:affiliated-implies-multiplication-CNP}, there is a function $h$ so that $Tf(x)=h(x)f(x)$ for all $f\in \nbdom T $. We call the function $h$ the {\em symbol} of $T$.   We will henceforth refer to $h$ simply as a {\em densely defined multiplier} (with the understanding that it is always a {\em closed} operator). 
Also, since $T$ is densely defined and closed, it has a densely defined adjoint $T^*$.  We recall that for any closed densely defined operator $T$, the graph of $T^*$ is related to the graph of $T$ via 
\begin{equation}
G(T^*) = JG(T)^\bot
\end{equation}
where $J$ is the $2\times 2$ block unitary
\begin{equation}
J=\begin{pmatrix} 0 & -1 \\ 1 & 0\end{pmatrix}.
\end{equation}

To say that $T$ is affiliated to $\mr{Mult} \, \cH$ means, as we have observed and used in the last section, that the graph $G(T)\subset \mathcal H\oplus \mathcal H$ is invariant for $M_\varphi\oplus M_\varphi$, for every $\varphi\in \mr{Mult} \, \cH$. It follows that the graph of $T^*$ is invariant for all the $M_\varphi^*\oplus M_\varphi^*$. This is then equivalent to saying that $\nbdom T ^*$ is invariant for all $M_\varphi^*$, and for all $f\in \nbdom T ^*$ and all $\varphi\in Mult (\mathcal H)$ we have
\[
  T^*M_\varphi^*f =M_\varphi^*T^*f.
\]

In this section we equip $\nbdom T $ and $\nbdom T ^*$ with norms making them into reproducing kernel Hilbert spaces, contractively contained in $\mathcal H$, and derive some basic properties of these spaces. 

As usual for closed operators, we can equip the domain $\nbdom T$ with the {\em graph norm}
\begin{equation}\label{eqn:graph-norm-def-CNP}
  \|f\|_{\nbdom T}^2 = \|f\|^2_{\mathcal H} +\|Tf\|^2_{\mathcal H}.
\end{equation}
This norm is just what makes the natural bijection $f\to (f, Tf)$ from $\nbdom T$ to $G(T)\subset \mathcal H \oplus \mathcal H$ a unitary operator. In particular since $T$ is assumed closed, $G(T)$ is closed in $\mathcal H\oplus \mathcal H$ and therefore $\nbdom T$ is complete with this norm. Of course, the corresponding statements also hold for $T^*$ and $\nbdom T ^*$.

 From the definition of the graph norm (\ref{eqn:graph-norm-def-CNP}), the inclusion maps $\nbdom T \subset \mathcal H$ and $\nbdom T ^*\subset \mathcal H$ are
contractive linear operators. This implies that $\nbdom T$ and $\nbdom T ^*$ are reproducing kernel Hilbert spaces over $X$; we write
$k^T(x,y)$ and $k^{T^*}(x,y)$ for their kernels, respectively. We can use the description of the graph of $T$ obtained in the last section (equation \ref{eqn:beurling-graph-CNP}) to compute $k^T$ and $k^{T^*}$. In particular we fix the auxiliary Hilbert space $\mathcal E$ and the contractive multiplier
\begin{equation}
\Theta=\begin{pmatrix} A \\ B\end{pmatrix} :\mathcal H\otimes \mathcal E\to \mathcal H\otimes \mathbb C^2
\end{equation}
so that $G(T)=\nbran (M_\Theta)$. We call $A,B$ the {\em Beurling pair} associated to $T$ (or, equivalently, to its symbol $h$). By the McCullough-Trent Beurling theorem for CNP spaces \cite{mccullough-trent-2000}, the row multipliers $A$ and $B$ are unique up to the choice of basis in $\mathcal E$. Precisely, if $A_1, B_1$ is another Beurling pair, there is a unitary transformation $U:\mathcal E\to \mathcal E$ such that $A_1(x) = A(x)U$ and $B_1(x)=B(x)U$ for all $x\in X$. For given $h$, we will fix one choice of Beurling pair (corresponding to a choice of orthonormal basis for $\mathcal E$) and refer to it as {\em the} Beurling pair.  In particular, if $A,B$ and $A_1, B_1$ are two different Beurling pairs, then for all $x,y\in X$ we have
\[
  A(x)A(y)^* =A_1(x)A_1(y)^*
\]
and similarly for $B, B_1$, so that the formulas below do not actually depend on the choice of Beurling pair.

\begin{thm}
Let $h$ be a densely defined multiplier, $T$ the operator $Tf=hf$, and $A,B$ the Beurling pair for $h$. Then the reproducing kernels for $\nbdom T$ and $\nbdom T_*$ are given by
\begin{equation}
k^T(x,y) = A(x)A(y)^*k(x,y)
\end{equation}
and
\begin{equation}
k^{T^*}(x,y) = (1-B(x)B(y)^*)k(x,y)
\end{equation}
respectively. 
\end{thm}

\begin{proof}
We work with $k^T$ first. Since $M_\Theta$ is a partial isometry with range $G(T)$, the operator $P_{G(T)}:=M_\Theta M_\Theta^*$ is the orthogonal projection from $\mathcal H\oplus \mathcal H$ onto the graph $G(T)$.  On the one hand, by the reproducing property of $k^T$ we have for each $y\in X$ and each $f\in \nbdom T$
\begin{equation}
f(y) = \langle f, k^T_y\rangle_{\nbdom T}.
\end{equation}
On the other hand, since the map $f\to \begin{pmatrix} f \\ Tf \end{pmatrix}$ is a unitary from $\nbdom T$ onto $G(T)$, we have for any $f, g\in \nbdom T$
\begin{equation}
\langle f,g\rangle_{\nbdom T} = \left\langle \begin{pmatrix} f \\ Tf \end{pmatrix},  \begin{pmatrix} g \\ Tg\end{pmatrix} \right\rangle_{\mathcal H\oplus \mathcal H}
\end{equation}
Next, for any $f\in\nbdom T$ and $y\in X$, we have
\begin{align}
f(y) &= \left\langle \begin{pmatrix} f \\ Tf \end{pmatrix},  \begin{pmatrix} k_y \\  0 \end{pmatrix} \right\rangle_{\mathcal H\oplus\mathcal H} \\
 &= \left\langle \begin{pmatrix} f \\ Tf \end{pmatrix}, P_{G(T)} \begin{pmatrix} k_y \\  0 \end{pmatrix} \right\rangle_{\mathcal H\oplus \mathcal H} 
\end{align}
It follows that $k^T_y$ is obtained by projecting $\begin{pmatrix} k_y \\ 0\end{pmatrix}$ into $G(T)$ and reading off the first entry. We get
\begin{equation}
 P_{G(T)} \begin{pmatrix} k_y \\  0 \end{pmatrix} = M_\Theta M_\Theta^*\begin{pmatrix} k_y \\  0 \end{pmatrix}  =\begin{pmatrix} M_AM_A^* & M_AM_B^* \\ M_BM_A^* & M_BM_B^*\end{pmatrix} \begin{pmatrix} k_y \\  0 \end{pmatrix} = \begin{pmatrix} AA(y)^* k_y \\ BA(y)^*k_y\end{pmatrix}
\end{equation}
so that
\begin{equation}
k^T(x,y) = A(x)A(y)^*k(x,y)
\end{equation}
as desired. 

The kernel $k^{T^*}$ is obtained similarly, using the relationship $G(T^*)= JG(T)^\bot$ mentioned above.  Thus we may calculate the projection onto the graph of $T^*$ as
\begin{align}
  P_{G(T^*)} &= J^*(I-M_\Theta M_\Theta^*)J \\
               &= \begin{pmatrix} 0 & 1 \\ -1 & 0 \end{pmatrix} \begin{pmatrix}
    I-M_AM_A^* & -M_AM_B^* \\ -M_BM_A^* & I-
    M_BM_B^*\end{pmatrix}\begin{pmatrix} 0 & -1 \\ 1 & 0 \end{pmatrix} \\
  &= \begin{pmatrix}
    I-M_BM_B^* & M_BM_A^* \\ M_AM_B^* & I-
    M_AM_A^*\end{pmatrix}
\end{align}
The claimed formula for $k^{T^*}$ now follows by the same reasoning as before. 
\end{proof}

By their definition, we conclude that the spaces $\nbdom T$ and $\nbdom T^*$ coincide (isometrically) with the deBranges-Rovnyak operator range spaces $\mathcal M(A)$ and $\mathcal H(B)$ respectively. 

In a similar way the ranges $\nbran T$ and $\nbran T^*$ can be made into reproducing kernel Hilbert spaces: indeed, each range can be realized as the image of a contractive map from the graph of the operator back into $\mathcal H$:
\[
  \begin{pmatrix} f \\ g\end{pmatrix} \to g.
\]
Using the formula for the projection onto $G(T)$ as in the above proof, one can show that $\nbran T = \mathcal M(B)$ and $\nbran T^*=\mathcal H(A)$. Since we will not have cause to deal with these spaces as such, we omit the details.

\subsection{The Drury-Arveson space}
In the special case of the Drury-Arveson space $H^2_d$, we can say more about the Beurling pair $A,B$, in particular we have an analogue of the ``Pythagorean mate'' condition (\ref{eqn:pythagorean-mate}) described in the introduction.

Since we have exhibited $G(T)$ as the range of a partially isometric
multiplier, applying the theorem of Greene-Richter-Sundberg \cite{greene-richter-sundberg-2002} we
conclude that 
\begin{equation}
\begin{pmatrix} A(\zeta) \\ B(\zeta)\end{pmatrix}  
\end{equation}
is a partial isometry for almost every $\zeta\in\partial\mathbb B^d$; in
particular we conclude the $2\times 2$ matrix
\begin{equation}
\begin{pmatrix} A(\zeta)A(\zeta)^* & A(\zeta)B(\zeta)^*   \\ B(\zeta)
  A(\zeta)^* & B(\zeta) B(\zeta)^*\end{pmatrix}
\end{equation}
is a projection in $\mathbb C^2$ for almost every $\zeta$, and (again
by their theorem) the rank of these projections is a.e. constant. On
the other hand, since $B=hA$, the rank must be one, and since the
trace of a projection is its rank we conclude that
\begin{equation}\label{eqn:inner-seq}
  \|A(\zeta)\|^2+\|B(\zeta)\|^2 =1  
\end{equation}
for a.e. $\zeta\in\partial\mathbb B^d$. (This should be compared with the ``Pythagorean mate'' relation (\ref{eqn:pythagorean-mate})). Thus, for any function
$h$ defining a closed, densely defined multiplier, there exist $A, B$
as above so that 
\begin{equation}\label{eqn:h=B/A}
  |h(\zeta)|^2 = \frac{\|B(\zeta)\|^2}{\|A(\zeta)\|^2} \quad
    \text{a.e. } \zeta
\end{equation}
with $A, B$ satisfying (\ref{eqn:inner-seq}).  
\begin{cor}
  If $h$ is a densely defined multiplier of $H^2_d$, there exist
  contractive multiplier sequences
  \begin{equation}
    A=(a_1, a_2, \dots a_n, \dots) \quad \text{and}\quad B=(b_1, b_2,
    \dots b_n,\dots)    
  \end{equation}
such that for a.e. $\zeta\in\partial\mathbb B^d$
\begin{equation}
  \|A(\zeta)\|^2  = \frac{1}{1+|h(\zeta)|^2}, \quad   \|B(\zeta)\|^2 = \frac{|h(\zeta)|^2}{1+|h(\zeta)|^2}.
\end{equation}
\end{cor}
\begin{proof}
This is immediate, on combining (\ref{eqn:inner-seq}) and (\ref{eqn:h=B/A}).
\end{proof}

\subsection{Density of kernel functions in $\nbdom T ^*$}\label{sec:density}
As observed in the introduction, it is possible to prove many of the classical facts about the non-extreme $\cH(b)$ spaces in the disk by identifying them with $\nbdom T^*$ for a densely defined multiplier $T$. In particular, a basic fact about the $\cH(b)$ spaces in the disk is that they contain the polynomials and Szeg\H{o} kernels, and each of these systems of functions span a dense linear subspace of $\cH (b)$. For the non-extreme $\cH b)$ subspaces of the Drury-Arveson space, it was shown in \cite{jury-2014} that these spaces contain the polynomials, and the question was posed of whether or not the polynomials are dense. This question is still open. However if instead we regard $\nbdom T^*$ as the proper analog of the non-extreme $\mathcal H(b)$ spaces, the question will have a positive answer. In this section we take up the case of the density of the kernel functions in the case of $\nbdom T^*$ in a general CNP space; the result for polynomials in the case of $H^2_d$ may be obtained by similar reasoning, but we will deduce it instead from a stronger result in Section~\ref{sec:invariant-subspaces} (and give a second, more constructive proof in Section~\ref{sec:constructive}).

Let $h$ be an
unbounded multiplier, and let $T$ the operator of multiplication by $h$
with dense domain $\nbdom T$ as above. We have row multipliers
$A=(a_1\ a_2\ \dots), B=(b_1,\ b_2, \dots)$ such that 
\begin{equation}
  G(T) = \nbran \begin{pmatrix} M_A \\ M_B\end{pmatrix}
\end{equation}
and hence also 
\begin{equation}
  G(T^*)^\bot   = \nbran \begin{pmatrix} M_B \\ -M_A\end{pmatrix}.
\end{equation}
We have also seen that $\nbdom T ^*$ (equipped with the
graph norm) is isometrically the
space $\mathcal H(B)$, via the map from $\nbdom T^*$ to $G(T^*)$
given by
\begin{equation}
  f\to \begin{pmatrix} f\\ T^*f\end{pmatrix}.  
\end{equation}
Fundamental to the study of multiplication operators in reproducing kernel Hilbert spaces is the eigenfunction property of the kernels, namely for every $x\in X$ and $\varphi\in \mr{Mult} \, \cH$ we have
\[
  M_\varphi^* k_x=\varphi(x)^*k_x.
\]
We begin this section by showing this extends to the unbounded case. 

\begin{lemma}\label{lem:DA-kernel-formulas}
  Let $\mathcal H$ be a CNP kernel space on a set $X$ with kernel $k$, and $h$ an unbounded multiplier of $\mathcal H$. Then for
  each $x\in X$, we have $k_x\in \nbdom T^*=\mathcal H(B)$ with 
  \begin{equation}
    T^*k_x = h(x)^*k_x.    
  \end{equation}
Moreover for all $x\in X$ and all $f\in \mathcal H(B)$
\begin{equation}
  \langle f,k_x\rangle_{\mathcal H(B)} = f(x) + h(x)(T^*f)(x)  
\end{equation}
so in particular
\begin{equation}
  \langle k_y, k_x\rangle_ {\mathcal H(B)} = k(x,y) +h(x)h(y)^*k(x,y)
\end{equation}
\end{lemma}
\begin{proof}
  Fix a point $x\in X$. Then for any function $f\in \nbdom T $, 
  \begin{equation}
    |\langle Tf,k_x\rangle | =|h(x)f(x)|\leq |h(x)|\|k_x\| \|f\|,     
  \end{equation}
which shows that $k_x\in \nbdom T^*$. It is then immediate that
$\langle f, T^*k_x\rangle = \langle Tf, k_x\rangle = \langle
hf,k_x\rangle$ and hence $T^*k_x = h(x)^*k_x$.  

For the second claim, identifying $\mathcal H(B)$ with the graph of
$T^*$ we obtain
\begin{align}
  \langle f, k_x\rangle_{\mathcal H(B)} &=
  \left\langle \begin{pmatrix}f \\ T^*f \end{pmatrix} , \begin{pmatrix}
      k_x \\ h(x)^*k_x\end{pmatrix}\right\rangle \\
&= \langle f,k_x\rangle + \langle T^*f, h(x)^*k_x\rangle \\
&= f(x) +h(x)(T^*f)(x).
  \end{align}
\end{proof}
\begin{prop}
  The span of the kernel functions $k_x$, as $x$ ranges over
  $X$, is dense in $\mathcal H(B) \cong \nbdom T ^*$. 
\end{prop}
\begin{proof}
  Suppose $f\in \mathcal H(B)$ is orthogonal to all the $k_x$. By the
  lemma we then have 
  \begin{equation}
    0=\langle f, k_x\rangle_{\mathcal H(B)} = f(x) +h(x)(T^*f)(x)    
  \end{equation}
for all $x\in X$, hence 
\begin{equation}\label{eqn:TT*f}
  h(x) (T^*f)(x) = -f(x).  
\end{equation}
But $f$ is in $\mathcal H(B)$, hence in $\mathcal H$, so $h(T^*f)$ is in
$\mathcal H$ and therefore $T^*f$ belongs to $\nbdom T $ (by the definition of
this domain).  Thus we can rewrite (\ref{eqn:TT*f}) as
\[
  TT^*f =-f,
\]
so taking the inner product with $f$, and again using the fact that $f\in \nbdom T ^*$, we get
\[
  -\|f\|^2 = \langle TT^*f, f\rangle = \langle T^*f, T^*f\rangle = \|T^*f\|^2\geq 0,
\]
whence $f=0$. 

As a consequence, we see that the span of the kernel functions $\{k_x\}_{x\in X}\subset \nbdom T ^*$ is a {\em core} for the operator $T^*$. 
\end{proof}

\section{Representing pairs and containments}\label{sec:representing-pairs}

For a densely defined multiplier $h$, we have seen from the proof of Theorem~\ref{thm:affiliated-implies-multiplication-CNP} that its graph is the
range of a $2\times N$ (generically $N=\infty$) partially isometric
multiplier
\begin{equation*}
  \Theta=\begin{pmatrix} A \\ B\end{pmatrix}.
\end{equation*}
On the other hand, we also proved in Theorem~\ref{thm:affiliated-implies-multiplication-CNP} that there exist (non-unique) scalar multipliers $a,b$ so that $b=ah$. We now describe the relationship between the Beurling pair $A,B$ and the scalar multipliers $a,b$.

\begin{defn}
Let $h$ be a densely defined multiplier. We say a pair of bounded (nonzero) multipliers $a,b\in \mr{Mult} \, \cH$ is a {\em representing pair} for $h$ if:
\begin{itemize}
\item[i)]  $b=ah$, and
\item[ii)]   the column $\theta=\begin{pmatrix} a \\ b\end{pmatrix}$ is contractive.
\end{itemize}
\end{defn}
 The condition (ii) is just a normalizing condition that will sometimes be convenient. The following theorem characterizes all representing pairs in terms of the Beurling pair $A,B$. 

 For a contractive multiplier $b\in \mr{Mult} \, \cH$, we let $\mathcal H(b)$ denote the reproducing kernel Hilbert space with kernel
 \[
   k^b(x,y) = (1-b(x)b(y)^*)k(x,y).
 \]
 and for any $a\in \mr{Mult} \, \cH$ let $\mathcal M(a)$ be the reproducing kernel Hilbert space with kernel
 \[
   k_a(x,y) = a(x)a(y)^* k(x,y).
 \]
 The spaces $\mathcal H(b)$, as a set, is equal to the range of the operator $(I-M_bM_b^*)^{1/2}$, re-normed so as to make this operator a partial isometry of $\mathcal H$ onto $\mathcal H(b)$. Similarly, $\mathcal M(a)$ is a re-normed image of $(M_aM_a^*)^{1/2}$. See the introduction to Sarason's book \cite{sarason-1994} for generic facts about $\mathcal H(b)$, $\mathcal M(a)$.

 In particular, the space $\mathcal H(b)$ is contractively contained in $\mathcal H$. The next proposition gives a parameterization of all the (contractive) representing pairs $(a,b)$; they are all ``subordinate'' to the Beurling pair $A,B$.

 \begin{prop}\label{prop:big-phi-little-phi-CNP}
   Let $h$ be a densely defined multiplier of the $CNP$ space $\mathcal H$, with Beurling pair $A,B\in (\mr{Mult} \, \mathcal H\otimes \mathcal E, \mathcal H)$. Then for every representing pair $(a,b)$, there exists a contractive multiplier $\Phi\in (\mr{Mult} \, \mathcal H, \mathcal H\otimes \mathcal E)$ such that
   \[
     a=A\Phi\quad \text{and} \quad b=B\Phi.
   \]
   Moreover, for every representing pair $(a,b)$ we have contractive containments
   \[
     \mathcal M(a)\subset \mathcal M(A) \quad \text{and} \quad \mathcal H(B)\subset \mathcal H(b)
   \]
 \end{prop}
 
\begin{proof}
  By its definition, for the Beurling pair $A,B$ we have that the multiplier
  \[
    M_\Theta= \begin{pmatrix}M_A \\ M_B\end{pmatrix}
  \]
  is a partial isometry mapping $\mathcal H\otimes \mathcal E$ onto the graph $G(T)$. Thus, the operator $M_\Theta M_\Theta^*$ is the orthogonal projection of $\mathcal H\oplus \mathcal H$ onto $G(T)$. If we fix a (contractive) representing pair $(a,b)$ and form the column multiplier
  \[
    M_\theta = \begin{pmatrix}M_a \\ M_b\end{pmatrix},
  \]
  then the operator $M_\theta M_\theta^*$ is a contractive map from $\mathcal H$ into $\mathcal H\oplus \mathcal H$ whose range is contained in the range of the projection $M_\Theta M_\Theta^*$, and therefore we have the inequality of positive operators
  \begin{equation}
    M_\theta M_\theta^*\leq M_\Theta M_\Theta^*. 
  \end{equation}
  This entails (and is really equivalent to) the contractive containment of the space $\mathcal M(M_\theta)\subset \mathcal M(M_\Theta)$, and is also equivalent to the positivity of the $2\times 2$ matrix valued kernel
  \begin{equation}\label{eqn:theta-kernel}
    (\Theta(x)\Theta(y)^* - \theta(x)\theta(y)^*)k(x,y).
  \end{equation}
  By the general form of the Nevanlinna-Pick interpolation theorem in CNP spaces \cite{ball-trent-vinnikov-2001}, from the positivity of this kernel we deduce the existence of a contractive multiplier $\Phi\in \mr{Mult} \, (\mathcal H, \mathcal H\otimes \mathcal E)$ so that $\theta = \Theta \Phi$.  This proves the first statement of the theorem.

  The diagonal entries of the kernel (\ref{eqn:theta-kernel}) are
  \begin{equation}\label{eqn:a-kernel}
    (A(x)A(y)^*-a(x)a(y)^*)k(x,y)
  \end{equation}
  and
  \begin{equation}\label{eqn:b-kernel}
    (B(x)B(y)^* -b(x)b(y)^*)k(x,y)
  \end{equation}
  The positivity of (\ref{eqn:a-kernel}) expresses the contractive containment $\mathcal M(a)\subset \mathcal M(A)$, while the positivity of (\ref{eqn:b-kernel}) is equivalent to the positivity of
  \[
    (1-b(x)b(y)^*)k(x,y) - (1-B(x)B(y)^*)k(x,y),
  \]
  which is the contractive containment $\mathcal H(B)\subset \mathcal H(b)$.
\end{proof}
{\bf Remark:} Later in this section we will prove more refined relationships between the spaces $\mathcal M(A)$, $\mathcal H(B)$ and the families $\mathcal M(a), \mathcal H(b)$ respectively.

{\bf Remark:} In general, for the column partial isometry $\Theta = \begin{pmatrix} A \\ B\end{pmatrix}$, the row $\Theta^t=\begin{pmatrix} A & B\end{pmatrix}$ is {\em not} contractive. To say that $\Theta^t=\begin{pmatrix} A & B\end{pmatrix}$ is contractive means that the kernel 
      \begin{equation}
        (1-B(x)B(y)^* - A(x)A(y)^*)k(x,y)
      \end{equation}
is positive. This kernel is positive if and only if there is a contractive containment $\mathcal H(B)\subset \mathcal M(A)$, or, by the way $A$ and $B$ are constructed from the function $h$, that $\nbdom T ^*\subset \nbdom T$ contractively. This in turn means that for all $f\in \nbdom T$, by the definition of the graph norms
\begin{equation}
  \|f\|^2+\|T^*f\|^2\leq \|f\|^2 +\|Tf\|^2, \quad \text{ or simply }\quad  \|T^*f\|^2\leq \|Tf\|^2
\end{equation}
But if we choose $h$ to be a bounded multiplier to start with, this is just equivalent to the statement that the multiplication operator $T$ is hyponormal, which is typically false for multipliers of CNP spaces; for example in the two-variable Drury-Arveson space $H^2_d$ the coordinate multipliers $z_1$, $z_2$ already fail to be hyponormal. Generically, multipliers of CNP spaces will not be hyponormal except in special cases such as $H^2$ of the disk; see \cite{hartz-2015}.

\begin{lemma}\label{lem:b*f=a*g-CNP}
  If $a,b$ is a representing pair for the densely defined multiplier $h$, then for every $f\in \nbdom T ^*$ there exists a function $g\in \mathcal H$ such that
  \begin{equation}\label{eqn:b*f=a*g-CNP}
    M_b^*f =M_a^* g.
  \end{equation}
More generally, a function $f$ belongs to $\nbdom T ^*$ if and only if there exists a function $g\in \mathcal H$ such that (\ref{eqn:b*f=a*g-CNP}) holds for {\em every} representing pair $a,b$; when such $g$ exists it is unique and given by $g=T^*f$. 
\end{lemma}
\begin{proof}
  Indeed, fix $f$ and a representing pair, putting $g=T^*f$ gives (\ref{eqn:b*f=a*g-CNP}).  Conversely, for a given $f$, if there exists a single $g$ that works for all representing pairs, then it works for all coordinates $a_j, b_j$ of the Beurling pair $A,B$, which means that
  \begin{equation}\label{eqn:B*f=A*g}
    M_B^*f = M_A^*g.
  \end{equation}
Recalling that the graph of $T$ is the range of the column $\begin{pmatrix} M_A \\ M_B\end{pmatrix}$, and that $G(T^*) = JG(T)^\bot$, we see that the relation (\ref{eqn:B*f=A*g}) holds precisely when $\begin{pmatrix} f \\ g\end{pmatrix}$ belongs to the graph of $T^*$, i.e. precisely when $f\in \nbdom T ^*$ and $g=T^*f$. 
\end{proof}
{\bf Remark:} Of course, for particular choices of representing pair $(a,b)$, there can be many solutions $g$ to $M_b^*f=M_a^*g$; indeed, if $a$ is not a cyclic multiplier (that is, if the range of $M_a$ is not dense, which it need not be) then $\nbker M_a^*$ is nontrivial, so that if $g$ solves $M_b^*f=M_a^*g$, then so does $g+g_0$ for any $g_0\in \nbker M_a^*$.

\begin{lemma}\label{lem:containment-Ma*}
  If $a,b$ is a representing pair for the densely defined multiplier $h$, then $\mathcal M(a^*)$ is contained contractively in  $\mathcal H(B)$. Moreover, if $a$ is nonvanishing, then $\mathcal M(a^*)$ is dense in $\mathcal H(B)$. 
\end{lemma}
\begin{proof}
  Let $f\in \nbdom T$, let $a,b$ be a representing pair for $h$, and $g\in \nbker M_a^{*\bot}$. Then
  \begin{equation}
    |\langle Tf, M_a^*g\rangle| = |\langle M_aTf, g\rangle| = |\langle bf, g\rangle|\leq \|f\|_{\mathcal H}\|g\|_{\mathcal H} = \|f\|_{\mathcal H} \|M_a^*g\|_{\mathcal M(a^*)}
  \end{equation}
which shows that $M_a^*g\in \nbdom T ^*=\mathcal H(B)$. Moreover we see that since $ah=b$, we have $T^*M_a^*g=M_b^*g$. Finally, since the column $\begin{pmatrix} b\\ a\end{pmatrix}$ is assumed contractive, the main result of \cite{hartz-preprint} shows the row multiplier $\begin{pmatrix} b & a\end{pmatrix}$ is also contractive, so that
\begin{equation}
  \|M_a^*g\|^2_{\mathcal H(B)} = \|M_a^*g\|^2_{\mathcal H}+ \|M_b^*g\|^2_{\mathcal H} \leq \|g\|^2_{\mathcal H} =\|M_a^*g\|_{\mathcal M(a^*)}^2
\end{equation}
which shows that the inclusion is contractive. 

To see that $\mathcal M(a^*)$ is dense, suppose $f\in\mathcal H(B)$ and $f\bot M_a^*g$ for all $g\in \mathcal H$. This means that for all $g$, 
\begin{align}
0 &= \left\langle \begin{pmatrix} f \\ T^*f \end{pmatrix} , \begin{pmatrix} M_a^*g \\ M_b^*g\end{pmatrix}\right\rangle  \\
&= \langle f, M_a^*g\rangle + \langle T^*f, M_b^*g\rangle \\
&= \langle af+ bT^*f, g\rangle
\end{align}
and thus $af+bT^*f \equiv 0$. Since $a$ is assumed nonvanishing, we deduce that $f+hT^*f\equiv 0$. This implies that $T^*f\in \nbdom T $, so that $f\in \nbdom TT^*$. Hence $(I+TT^*)f=0$, which forces $f=0$ since $I+TT^*$ is bounded below. 
\end{proof}

{\bf Remark:} In spaces of holomorphic functions on a connected domain, the hypothesis ``$a$ is nonvanishing'' in the last part of the lemma can be weakened to ``$a$ is not identically $0$,'' since the ring of holomorphic functions on a connected domain has no zero divisors.

\subsection{Relation of $\mathcal M(A)$ to $\mathcal M(a)$ and $\mathcal H(B)$ to $\mathcal H(b)$} 
Fix a densely defined multiplier $h$. In a representing pair $a,b$ for $h$, we call $b$ a {\em numerator} and $a$ a {\em denominator}.

We saw in the last subsection that $\mathcal M(A)$ contains all the $\mathcal M(a)$ contractively, and that $\mathcal H(B)$ is contained contractively in all the $\mathcal H(b)$. The next two propositions sharpen these statements. We shall use $ball(\cdot)$ to denote the closed unit ball of a Banach space. 

\begin{prop} Let $h$ be a densely defined multiplier with Beurling pair $A,B$. We have
  \[
    \mathcal M(A)=\bigcup_{a} \mathcal M(a)
  \]
  where the union is taken over all denominators $a$. More precisely, $f\in \mathcal M(A)$ with $\|f\|_{\mathcal M(A)} \leq 1$ if and only if there exists a numerator $a$ so that $f\in \mathcal M(a)$ and $\|f\|_{\mathcal M(a)}\leq 1$. Thus
  \[
    ball(\mathcal M(A)) = \bigcup_a ball(\mathcal M(a))
  \]
  where the union is taken over all numerators $a$. 
\end{prop}

\begin{proof} It suffices to prove the second statement, about the unit balls. We have already seen that each $\mathcal M(a)$ is contained contractively in $\mathcal M(A)$, which means precisely that $ball(\mathcal M(a))$ is contained in $ball(\mathcal M(A))$, for all denominators $a$.  For the reverse inclusion, suppose $f\in ball(\mathcal M(A))$. This means that there is a function $F\in \mathcal H\otimes \mathcal E$ with $\|F\|\leq 1$ so that $f=AF$. Now, by \cite[Theorem 1.1]{jury-martin-2019}, there exists a contractive multiplier $\Phi\in (\mr{Mult} \, \mathcal H, \mathcal H\otimes \mathcal E)$ and a cyclic vector $\widetilde f\in \mathcal H$, with $\|\widetilde f\|=\|f\|\leq 1$ so that $F=\Phi \widetilde f$. Putting $a=A\Phi$, we have by Proposition (\ref{prop:big-phi-little-phi-CNP}) that $a$ is a denominator for $h$ and $f=a\widetilde f$. Since $\|\widetilde f\|\leq 1$, we conclude that $f\in ball(\mathcal M(a))$. 
\end{proof}
\begin{prop} Let $f\in \mathcal H$. Then $f\in \mathcal H(B)$ if and only if
\begin{itemize}
\item[i)] $f\in \mathcal H(b)$ for every numerator $b$, and
\item[ii)] $\sup_b \|f\|_{\mathcal H(b)} <\infty.$
\end{itemize}
where the supremum in (ii) is taken over all numerators $b$ for $h$. When this happens, the supremum in (ii) is in fact equal to $\|f\|_{\mathcal H(B)}$.

In other words,
\begin{equation}\label{eqn:intersection-of-balls}
  ball(\mathcal H(B)) = \bigcap_b ball(\mathcal H(b))
\end{equation}
where the intersection is taken over all numerators $b$.
\end{prop}

\begin{proof}
From the deBranges-Rovnyak picture of the $\mathcal H(B)$ norms, we have by definition for any $f\in \mathcal H$
\begin{equation}\label{eqn:H(B)-dbr-norm-def}
\|f\|^2_{\mathcal H(B)} =\sup_{g\in\mathcal M(B)} \left\{ \|f+g\|^2_{\mathcal H} -\|g\|^2_{\mathcal M(B)} \right\}
\end{equation}
and
\begin{equation}
\|f\|^2_{\mathcal H(b)}=\sup_{g\in\mathcal M(b)} \left\{ \|f+g\|^2_{\mathcal H} -\|g\|^2_{\mathcal M(b)}\right\}
\end{equation}
with the understanding that $f$ belongs to the space in question if and only if the supremum on the right hand side is finite. 

We have already shown that $\mathcal H(B)$ is contained contractively in $\mathcal H(b)$ for every numerator $b$. Conversely, fix a function $f$ and suppose (i) and (ii) hold. We must show that the supremum on the right hand side of (\ref{eqn:H(B)-dbr-norm-def}) is finite and equal to the supremum in (ii). To do this, fix a function $g\in\mathcal M(B)$; then by the definition of the $\mathcal M(B)$ space there is a function $G\in \mathcal H\otimes\mathcal E$ 
such that $g=BG$ and $\|g\|^2_{\mathcal M(B)}$. Again by the factorization theorem for sequences of $\mathcal H$ functions \cite[Theorem 1.1]{jury-martin-2019}, there is a cyclic vector $\widetilde{g}\in \mathcal H$ with $\|\widetilde{g}\|_{\mathcal H} = \|G\|_{\mathcal H}$ and a contractive column multiplier $\Phi$ so that $G =\Phi \widetilde g$. Put $b=B\Phi$, then $b$ is a numerator for $h$, $g=b\widetilde{g}$, and
\begin{align}
  \|f+g\|^2_{\mathcal H} -\|g\|^2_{\mathcal M(B)} &=  \|f+BG\|^2_{\mathcal H} -\|BG\|^2_{\mathcal M(B)} \\
  &= \|f+BG\|^2_{\mathcal H} -\|G\|^2_{\mathcal H} \\
  &= \|f+b\widetilde{g}\|^2_{\mathcal H} -\|\widetilde{g}\|_{\mathcal H} \\
  &\leq \|f\|_{\mathcal H(b)}^2.
\end{align}
In other words, for every $g\in \mathcal M(B)$ there exists a numerator $b$ for $h$ such that
\begin{equation}
  \|f+g\|^2_{\mathcal H} -\|g\|^2_{\mathcal M(B)} \leq  \|f\|_{\mathcal H(b)}^2.
\end{equation}
It follows that $\|f\|_{\mathcal H(B)} \leq \sup_b \|f\|_{\mathcal H(b)}$ as desired. 
\end{proof}

{\bf Remark:} Item (i) asserts the containment $\mathcal H(B)\subset \bigcap_b \mathcal H(b)$. We do not know if equality holds in general, this would be equivalent to removing the uniformity condition (ii).

\section{Similarity of restrictions, and corona pairs}\label{sec:corona}
Sarason \cite{sarason-1986b} proved, among other things, that if $b$ is not an extreme point of $ball(H^\infty)$ and $a$ is the outer function satisfying $|a|^2+|b|^2=1$ on the circle, then the backward shift $S^*$ restricted to $\mathcal H(b)$ is similar to the usual backward shift on $H^2$ if and only if $a,b$ is a corona pair. We can prove an analog of this result in CNP spaces, if we replace $\mathcal H(b)$ with $\nbdom T=\mathcal H(B)$. 

\begin{defn} We say that a pair of multipliers $a,b\in \mr{Mult} \, \cH$ is a {\em corona pair} if there exist multipliers $u,v\in \mr{Mult} \, \cH$ so that
  \begin{equation}\label{eqn:corona}
    au+bv\equiv 1.
  \end{equation}
\end{defn}
This is equivalent to saying that the row $(u\ v)$ is a left inverse to the column $\displaystyle{\begin{pmatrix} a\\ b\end{pmatrix}}$ (or that the column $\displaystyle{\begin{pmatrix} u\\ v\end{pmatrix}}$ is a right inverse to the row $(a\ b)$). 

\begin{thm}\label{thm:similar-to-S*} Let $h$ be a densely defined multiplier of the CNP space $\mathcal H$, with Beurling pair $(A, B)$.  The following are equivalent:
  \begin{itemize}
  \item[1)] There is an invertible map $W:\mathcal H\to \mathcal H(B)$ such that for every $\varphi\in \mr{Mult} \, \cH$, we have
    \begin{equation}\label{eqn:similarity-of-restriction}
       M_\varphi^*|_{\mathcal H(B)}  = WM_\varphi^*W^{-1}  
    \end{equation}
    
\item[2)] There exists a representing pair $a,b$ for $h$ such that $a$ is outer and $\mathcal H(B)=\mathcal M(a^*)$. 
\item[3)] There exists a representing pair $a,b$ for $h$ such that $a$ is outer and $a,b$ is a corona pair. 
  \end{itemize}
\end{thm}
\begin{proof}
  (1) implies (2): Suppose $W:\mathcal H\to \mathcal H(B)$ is invertible and (\ref{eqn:similarity-of-restriction}) holds for every $\varphi\in \mr{Mult} \, \cH$. Regarding $\mathcal H(B)$ as contractively included in $\mathcal H$, we may identify $W$ with a bounded operator commuting with all the $M_\varphi^*$ (over all multipliers $\varphi\in \mr{Mult} \, \cH$), and hence $W=M_a^*$ for some $a\in \mr{Mult} \, \cH$. Since $W$ is invertible, $M_a^*$ must be injective and hence $a$ is outer, and $\mathcal H(B)=\mathcal M(a^*)$ (with comparable norms, via the closed graph theorem).  To see that $a$ is part of a representing pair for $h$, it remains to show that $b:=ah$ is a (bounded) multiplier.  Since $W=M_a^*$ is bounded from $\mathcal H$ to $\mathcal H(B)=\nbdom T ^*$, we have for all $f\in \nbdom T$ and all $g\in \cH$
  \begin{equation}\label{eqn:ah-bounded-1}
    |\langle T f, M_a^*g\rangle| \leq \|f\|_{\cH}\|M_a^*g\|_{\nbdom T ^*} \leq \|W\|\|f\|_{\cH}\|g\|_{\cH}.
  \end{equation}
On the other hand
\begin{equation}\label{eqn:ah-bounded-2}
  \langle T f, M_a^*g\rangle =\langle ahf, g\rangle
\end{equation}
which combined with (\ref{eqn:ah-bounded-1}) shows that the map $f\to ahf$ is bounded on $\mathcal H$, that is, the multiplier $b=ah$ is bounded. Thus, $a,b$ will be a representing pair (after scaling to make it contractive). 

(2) implies (1): If $a$ is an outer multiplier and $\mathcal H(B)=\mathcal M(a^*)$, then $M_a^*$ implements a unitary equivalence between $M_\varphi^*$ (acting in $\mathcal H$) and $M_\varphi^*$ (acting in $\mathcal M(a^*)$). Since whenever $\mathcal H(B)=\mathcal M(a^*)$ holds, it holds with the two norms being comparable, we conclude that $W=M_a^*$ is a similarity and that (\ref{eqn:similarity-of-restriction}) holds. 

(2) implies (3): If we suppose that $\mathcal M(a^*)=\mathcal H(B)$, then the graph of $T^*$ is equal to the range of the column 
\begin{equation} 
M_\psi^* := \begin{pmatrix} M_a^* \\ M_b^* \end{pmatrix}.
\end{equation}
Since $T^*$ is a closed operator,  this says that the column $M_\psi^*$ has closed range. Since $a$ is assumed outer, $M_\psi^*$ is also injective, and hence bounded below. So, there is a constant $c$ so that $M_\psi M_\psi^*-c^2\geq 0$, which means that the kernel
\begin{equation}
(a(x)a(y)^* +b(x)b(y)^* -c^2)k(x,y)
\end{equation}
is positive. By the Pick interpolation theorem, there then exist multipliers $u,v$ so that
\begin{equation}
\begin{pmatrix} M_u^* & M_v^* \end{pmatrix} \begin{pmatrix} M_a^* \\ M_b^*\end{pmatrix} =c
\end{equation}
so on replacing $u,v$ by $u/c, v/c$ we have
\begin{equation} 
ua+vb=1,
\end{equation}
and thus $a,b$ is a corona pair.

(3) implies (2): Fix a corona pair $a,b$ representing $h$, with $a$ outer. From Lemma~\ref{lem:containment-Ma*}, we have $\mathcal M(a^*)\subset \mathcal H(B)$ always, so it remains to prove the reverse inclusion. Let $f\in \mathcal H(B)$, then by Lemma~\ref{lem:b*f=a*g-CNP}, there exists a $g$ such that $M_b^*f=M_a^*g$. Let $u,v$ be the multipliers solving the corona problem (\ref{eqn:corona}) for the pair $a,b$. Then
  \begin{equation}
    f= M_u^*M_a^*f + M_v^*M_b^*f =M_u^*M_a^* f+ M_v^*M_a^*g = M_a^*(M_u^*f+M_v^*g),  
  \end{equation}
so $f\in \mathcal M(a^*)$. 
\end{proof}

{\bf Remark:} In Sarason's theorem in the disk case, there is a fourth equivalence, namely that $b$ is a multiplier of $\mathcal H(b)$. The natural analog in the present context would be:

\vskip.1in

{\em 4) There exists a representing pair} $a,b$ {\em with} $b$ {\em a multiplier of} $\mathcal H(B)$.   

\vskip.1in

We do not know if this item (4) is equivalent to any of (1)--(3). However it is easy to prove the following weaker claim: 
\begin{prop}
If $a,b$ is a representing pair for $h$ which satisfies (any of) the conditions of Theorem~\ref{thm:similar-to-S*}, then $b$ multiplies $\mathcal H(B)$ into $\mathcal H(b)$. 
\end{prop}
\begin{proof}
Since the column $(b \ a)^t$ is contractive, we have the operator inequality
\begin{equation}
M_a^*M_a\leq I-M_b^*M_b
\end{equation}
and therefore
\begin{align}
M_b M_a^*M_aM_b^* &\leq M_b(I-M_b^*M_b)M_b^* \\
&\leq I-M_bM_b^*
\end{align}
which means that $b$ multiplies $\mathcal M(a^*)$ into $\mathcal H(b)$ (regardless of any other assumptions on $a,b$). Under the conditions of Theorem~\ref{thm:similar-to-S*}, $\mathcal H(B)=\mathcal M(a^*)$, which proves the proposition. 
\end{proof}

\section{Invariant subspaces}\label{sec:invariant-subspaces}

Sarason \cite{sarason-1986b,sarason-1986a} studied backward-shift invariant subspaces of the non-extreme $\mathcal H(b)$ spaces. In our more general setting of CNP spaces, replacing $\mathcal H(b)$ with $\nbdom T^*$, the appropriate analog will be subspaces which are invariant for all the adjoint multipliers $M_\varphi^*$ as $\varphi$ ranges over $\mr{Mult} \, \cH$.  
\begin{defn}Say that a closed linear subspace $\mathcal J\subset \mathcal H(B)$ is {\em $*$-invariant} if
  \[
    M_\varphi^* \mathcal J \subset \mathcal J
  \]
  for all multipliers $\varphi\in \mr{Mult} \, \cH$ (here $M_\varphi^*$ is understood as an operator acting from $\mathcal H$ to $\mathcal H$). 
\end{defn}

We have already seen that each $\mathcal H(B)$ itself is invariant for all the $M_\varphi^*$. One way to generate examples of $*$-invariant subspaces of $\mathcal H(B)$ is to take a closed subspace $\mathcal K\subset \mathcal H$ which is invariant for all the $M_\varphi^*$ and form its intersection $\mathcal J=\mathcal K\cap \mathcal H(B)$; since the inclusion $\mathcal H(B)\subset \mathcal H$ is continuous, this intersection will automatically be closed in the $\mathcal H(B)$ norm. Our main theorem, echoing Sarason's, is that every $*$-invariant subspace arises in this way. First we need a lemma.

\begin{lemma}\label{lem:invariant-subspace} Suppose $h$ is a densely defined multiplier with Beurling pair $A,B$. Let $\mathcal J\subset \mathcal H(B)$ be a closed (in the $\mathcal H(B)$ norm) $*$-invariant subspace.  Then
\begin{equation}
\text{span} \{M_a^*\mathcal J: a \text{ is a denominator for } h\}
\end{equation}
is a dense subspace of $\mathcal J$.
\end{lemma}
\begin{proof}
Suppose $f\in \mathcal J$ and $f\bot M_a^*\mathcal J$ for all $a$. Now consider any fixed $a$. Then in particular, since $M_a^*\mathcal J$ is $*$-invariant, we have $f\bot M_\varphi^*M_a^*f$ for every $\varphi\in \mr{Mult} \, \cH$. Let $g=T^*f$, then $T^*M_\varphi^*M_a^*f=M_\varphi^*M_a^*g$. We then have for every multiplier $\varphi\in \mr{Mult} \, \cH$, by the identification of the $\mathcal H(B)$ norm with the graph norm, 
\begin{align}
0 &=\langle f, M_\varphi^*M_a^*f\rangle_{\mathcal H(B)} \\
&= \langle f, M_\varphi^*M_a^*f\rangle_{\mathcal H} +  \langle g,M_\varphi^*M_a^*g\rangle_{\mathcal H} \\
&= \langle \varphi af, f\rangle +\langle \varphi a g,g\rangle
\end{align}
It follows that 
\begin{equation}\label{eqn:vanishing-against-ideal}
  \langle \varphi af, f\rangle +\langle \varphi a g,g\rangle =0
\end{equation}
for every multiplier $\varphi$, and for every denominator $a$. Since the  weak-* closed ideal generated by the denominators $a$ is equal to all of $Mult (\mathcal H)$, (see Lemma~\ref{lem:approximate-unit} below) it then follows from (\ref{eqn:vanishing-against-ideal}) that
\[
  \langle \varphi f, f\rangle +\langle\varphi g,g\rangle =0
\]
for all multipliers $\varphi$, in particular for $\varphi=1$, so we conclude that $\|f\|^2+\|g\|^2=0$ and thus $f=0$.
\end{proof}

\begin{thm}
Let $h$ be a densely defined multiplier. The $*$-invariant subspaces of $\nbdom T^*=\mathcal H(B)$ are the intersections of $\mathcal H(B)$ with the $*$-invariant subspaces of $\mathcal H$. Precisely, if $\mathcal J$ is a closed $*$-invariant subspace of $\mathcal H(B)$ and $\mathcal K$ is the closure of $\mathcal J$ in $\mathcal H$, we have $\mathcal J = \mathcal K\cap \mathcal H(B)$.

In particular if $\cJ ' \subseteq \cH (B)$ is $*$-invariant but not necessarily closed, and $\cK '$ is the embedding of $\cJ '$ into $\cH$, then $\cK '$ is dense in $\mathcal H$ if and only if $\cJ '$ is dense in $\mathcal H(B)$, \emph{i.e.} if and only if $\cJ '$ is a core for $T^*$.
\end{thm}
\begin{proof} We use a slightly modified version of the argument of Sarason~\cite[Theorem 5]{sarason-1986b}, using the lemma just proved. Let $\mathcal J\subset \mathcal H(B)$ be $*$-invariant  and let $\mathcal K$ be the closure of $\mathcal J$ in $\mathcal H$. Then $\mathcal K$ is $*$-invariant. If $f\in \mathcal K$, there is a sequence $f_n$ from $\mathcal J$ converging to $f$ in the $\mathcal H$ norm. If $a,b$ is a representing pair coming from the Beurling pair $A,B$, then $M_a^*f_n\to M_a^*f$ in the norm of $\mathcal M(a^*)$, and hence also in the norm of $\mathcal H(B)$, since $\mathcal M(a^*)$ is contained boundedly in $\mathcal H(B)$ (Lemma~\ref{lem:containment-Ma*}). Since $M_a^*f_n$ belongs to $\mathcal J$ for each $n$, so does $M_a^*f$. Thus $M_a^*\mathcal K\subset \mathcal J$. By Lemma~\ref{lem:invariant-subspace}, the span of $M_a^*(\mathcal K \cap \mathcal H(B))$, over all denominators $a$, is dense in $\mathcal K\cap \mathcal H(B)$, so $\mathcal K\cap \mathcal H(B)\subset \mathcal J$. As the opposite inclusion is trivial, the theorem is proved. 
\end{proof}

\begin{cor} If $h$ is a densely defined multiplier of the Drury-Arveson space $H^2_d$, then $\nbdom T ^*=\mathcal H(B)$ contains the polynomials, which are dense in this space.
\end{cor}
\begin{proof}
  Let $h$ be a densely defined multiplier and fix a monomial $z^{\bf n}$. By the definition of the Drury-Arveson norm, for every $f\in \nbdom T$ the functional
  \[
    f\to \langle hf, z^{\bf n}\rangle
  \]
  depends only on a fixed, finite set (depending on ${\bf n}$) of Taylor coefficients of $f$. Since evaluation of Taylor coefficients is bounded for the Drury-Arveson norm, it follows that this functional extends to a bounded functional on $H^2_d$, which means that $z^{\bf n}$ belongs to $\nbdom T ^*=\mathcal H(B)$. Letting $\mathcal J$ be the closure of the polynomials in $\nbdom T ^*$, and observing that $M_\varphi^*p$ is a polynomial for every polynomial $p$ and multiplier $\varphi\in \mr{Mult} \, H^2 _d$, (and that the polynomials are dense in $H^2_d$), we conclude by the theorem that $\mathcal J=\mathcal H(B)$. 
  \end{proof}

  The proof of the corollary can evidently be adapted to other holomoprhic CNP spaces in which the monomials and polynomials are sufficiently well-behaved, e.g. the Dirichlet space in the unit disk. 

\subsection{Extensions and restrictions}

Under a mild additional assumption on $\multH$, namely that nonzero multipliers are injective operators, we can show that a (closed) densely defined multiplier affiliated to $\multH$ has no proper (closed) extensions or restrictions that are also affiliated to $\multH$. This hypothesis is of course satisfied in holomorphic spaces over connected domains, such as the Drury-Arveson space, Dirichlet space, etc. 

\begin{lemma}
Let $\mathcal H$ be a CNP space and suppose that every nonzero bounded multiplication operator $f\to \varphi f$ is injective. If $T\neq 0$ is a closed, densely defined operator affiliated to $\multH$, then $T$ is injective.
\end{lemma}
\begin{proof}
We know that $T$ is given by multiplication by some function $h$, and there exist nonzero multipliers $a,b$ so that $b=ah$. If $hf=0$ then $bf=0$ and thus $f=0$, since $M_b$ is injective.
\end{proof}

\begin{thm}
  Let $\mathcal H$ is a CNP space and suppose every nonzero bounded multiplier is injective.   If $T_1, T_2$ are nonzero, closed, densely defined multipliers of $\mathcal H$ and $T_2 \subseteq T_1$ then $T_2 =T_1$.
\end{thm}

\begin{proof}
Here, $T_2 \subseteq T_1$ means that $T_2$ is a closed restriction of $T_1$, \emph{i.e.}, $\dom{T_2} \subseteq \dom{T_1}$ and $T_2 f = T_1 f$ for all $f \in \dom{T_2}$. Hence $\ran{T_2} \subseteq \ran{T_1}$. By \cite[Theorem 2(3)]{DFL}, since $\ran{T_2} \subseteq \ran{T_1}$ and both $T_1, T_2$ are closed, there is a densely-defined operator $C$, with $\dom{C} = \dom{T_2}$ so that $T_2 = T_1 C$, and $\dom{C}$ is invariant for $\multH$. For any $f\in \dom{C} = \dom{T_2}$ and any $\varphi\in \multH$,
\ba T_1 C \varphi f & = & T_2 \varphi f= \varphi T_2 f \nn \\
& = & \varphi T_1 C f = T_1 \varphi C f\nn \ea 
This proves that $T_1 (\varphi C - C \varphi) f=0$.  Since $T_1$ is closed, we have $T = M_h$ for some densely defined multiplier $h$, and $T_1$ is injective by the last lemma, so 
$$\varphi C f = C \varphi f, $$ for all $f \in \dom{C}$ and all $\varphi\in\multH$, so $C$ is affiliated to $\multH$. By the closability property, $\ov{C}$ is a closed operator affiliated to the $\multH$ so that $\ov{C} = M_c$ acts as multiplication by some function $c$.

But then, for any $f \in \dom{T_2}$, since $T_2 f = T_1 f$, we have that if $T_1 = M_h$ and $T_2 = M_h | _{\nbdom T_2}$,
$$ (T_1 f) (x) = h(x) f(x) = (T_2 f) (x) = h(x) f(x) = (T_1 C f) (x) = h(x) c(x) f(x). $$ 
It follows that $h(x) f(x) ( 1 -c (x) ) = 0$, and this proves that $C = M_c = I$, so that $T_1 = T_2$.
\end{proof}

\begin{cor} \label{corecor}
Under the same hypotheses as the theorem, if $T$ is a densely defined multiplier, then any dense, $\multH-$invariant $\mc{C} \subseteq \dom{T}$ is a core for $T$, and any dense space $\mc{C} ^* \subseteq \dom{T^*}$ invariant for all $M_\varphi^*$ is a core for $T^*$. 
\end{cor}

\begin{proof}
Clearly $\check{T} := \ov{T | _{\mc{C}}}  \subseteq T$, and $\check{T}$ is a closed operator affiliated to $\multH$. By the last theorem $\check{T} =T$, so that $\mc{C}$ is a core.

Conversely if $\check{T} ^* := \ov{T^* | _{\mc{C} ^*}}$ then $\check{T} ^* \subseteq T^*$ so that $T \subseteq \check{T}$, and since $\dom{\check{T} ^*}$ is invariant for all the $M_\varphi^*$, $\dom{\check{T}}$ will be $\multH-$invariant. It follows again that $T, \check{T}$ are closed, and so by the previous theorem, again $\check{T} = T$ so that $\check{T}^* = T^*$ and $\mc{C} ^*$ is a core.
\end{proof}

Another consequence is the following: still, of course, assuming that nonzero bounded multipliers are injective, we see that if $h$ is a densely defined multiplier that happens to belong to the Smirnov class, then it has a representing pair $a,b$ so that $a$ is cyclic; which means that $a$ has dense range, so by the corollary we see that the range of $M_a$ is a core for $T$, that is, $\mathcal M(a)$ is dense in $\mathcal M(A)=Dom(T)$.

  \section{Constructive approximation}\label{sec:constructive}

  We have seen that for any $\mathcal H(B)$ space in a CNP space $\mathcal H$, the kernel functions $k_x\in \mathcal H$ belong to $\mathcal H(B)$, and their span is dense in $\mathcal H(B)$. Similarly, in the special case of $H^2_d$, by the results of the last section the polynomials are dense in $\mathcal H(B)$. A result of El-Fallah et al. \cite[Theorem 4.4]{elfallah-et-al-2016} gives, in the special case of deBranges-Rovnyak subspaces in the unit disk, a more constructive method of approximating $f\in \mathcal H(B)$ by polynomials in $\mathcal H(B)$. In this section we show that their argument can be adapted to prove a similar constructive approximation theorem in our more abstract setting. Rather than work with polynomial approximation, we will work with approximation by finite sums of kernel functions, but the argument is easily modified to handle constructive polynomial approximation in those spaces where it makes sense (such as the Drury-Arveson space, Dirichlet space, etc.; see the comments at the end of the section).

  Let us fix some notation: as always we have a CNP space $\mathcal H$, a densely defined multiplier $h$, and a Beurling pair $A,B$ for $h$. Associated to the outer row multiplier $A$ is the ideal
  \[
    \mathcal A = \{A\Phi : \Phi\in \mr{Mult} \, \mathcal H, \mathcal H\otimes \mathcal E \}\subset \mr{Mult} \, \cH.
  \]
  As we have already seen, $\mathcal A$ is norm dense in $\mathcal H$, since $A$ is row outer. This will mean that $\mathcal A$ is automatically weak-* dense in $\mr{Mult} \, \cH$, which follows either from \cite{davidson-ramsey-shalit-2015}, or from the following Kaplansky-style density result, which is a more or less trivial reformulation of \cite[Lemma 3.3]{AHMR-preprint} (see also Remark 3.5 following that lemma; our statement corresponds to the special case $m=1$). The important point in our reformulation is the observation that the approximating sequence of multipliers can be chosen from a given norm dense ideal, which follows immediately from an inspection of the proof. 
  \begin{lemma}\cite[Lemma 3.3]{AHMR-preprint}\label{lem:approximate-unit} Let $\mathcal H$ be a CNP space and let $\mathcal A\subset \mr{Mult} \, \cH$ be an ideal of $\mr{Mult} \, \cH$ which is norm dense in $\mathcal H$. Then $\mathcal A$ admits a contractive weak-* approximate unit, that is, there is a sequence $(a_n)\subset \mathcal A$ such that $\|a_n\|_{\mr{Mult} \, \cH}\leq 1$ for all $n$, and $a_n\to 1$ in the weak-* topology.  (So in particular, $\mathcal A$ is weak-* dense in $\mr{Mult} \, \cH$.)
  \end{lemma}

  Since the kernel functions are eigenfunctions for adjoints of multipliers, and weak-* convergence implies pointwise convergence, it follows immediately from the lemma that given any finite set of functions $g_1, \dots, g_m\in span\{k_x\}_{x\in X}$, and any $\epsilon>0$, there exists an $n$ sufficiently large so that
  \[
    \|g_j-M_{a_n}^*g_j\|_{\mathcal H}<\epsilon\quad \text{for each } j=1,\dots, m.
  \]
  
  Let us now give the constructive proof of the kernel approximation theorem:
  \begin{thm}\label{thm:kernel-constructive-approximation}
    The span of the kernel functions $\{k_x\}_{x\in X}$ is dense in $\mathcal H(B)$.
  \end{thm}
  \begin{proof}
    Let $f$ belong to $\mathcal H(B)$, which is the domain of $T^*$; there exist finite combinations of kernel functions $g_1$, $g_2$ so that both
    \[
      \|f-g_1\|_{\mathcal H}<\epsilon, \quad \|T^*f -g_2\|_{\mathcal H}<\epsilon.
    \]
        Using the lemma on weak-* approximate units (and the remark following it), we may select a multiplier $a\in \mathcal A$ so that $\|a\|_{\mr{Mult} \, \cH}\leq 1$ and
    \[
      \|g_1-M_a^*g_1\|_{\mathcal H}<\epsilon, \quad \text{and} \quad \|g_2-M_a^*g_2\|_{\mathcal H}<\epsilon.
    \]
    Since $a\in \mathcal A$, the function $b=ah$ is a multiplier, let $M$ be the norm of the column multiplier $\begin{pmatrix} b\\a\end{pmatrix}$. Observe that $M_a^*$ maps $\mathcal H$ boundedly into $\mathcal H(B)$, with norm $M$.  Finally, choose $g$, a finite linear combination of kernel functions, so that $\|f-g\|_\mathcal H<\epsilon/M$. 

    Let us first estimate $\|f-M_a^*f\|_{\mathcal H(B)}$. We have
    \begin{align*}
      \|f-M_a^*f\|_{\mathcal H} &\leq \|f-g_1\|_{\mathcal H} +\|g_1-M_a^*g_1\|_{\mathcal H} + \|M_a^*g_1-M_a^*f\|_{\mathcal H} \\
      &< \epsilon+\epsilon+\epsilon \\
      &=3\epsilon,
    \end{align*}
    by the choice of $g_1$ and $a$. Similarly we have for $T^*f$
    \begin{align*}
      \|T^*f-M_a^*T^*f\|_{\mathcal H} &\leq \|T^*f-g_2\|_{\mathcal H} +\|g_2-M_a^*g_2\|_{\mathcal H} + \|M_a^*g_2-M_a^*T^*\|_{\mathcal H} \\
      &< \epsilon+\epsilon+\epsilon \\
      &=3\epsilon. 
    \end{align*}
    Thus altogether we have  
    \begin{align*}
      \|f-M_a^*f\|^2_{\mathcal H(B)} &= \|f-M_a^*f\|_{\mathcal H}^2 + \|T^*f-M_a^*T^*f\|_{\mathcal H}^2 \\
      &< 18\epsilon^2,
    \end{align*}
    so $\|f-M_a^*f\|_{\mathcal H(B)}<5\epsilon$. 

    Finally,
    \begin{align*}
      \|f-M_a^*g\|_{\mathcal H(B)} &\leq \|f-M_a^*f\|_{\mathcal H(B)} + \|M_a^*f-M_a^*g\|_{\mathcal H(B)}\\
                                   &< 5\epsilon + M(\epsilon/M) \\
                                   &=6\epsilon.
    \end{align*}
    Since $g$ was a finite linear combination of kernel functions, so is $M_a^*g$, and we are done. 
  \end{proof}
  {\bf Remark:} The same argument can be seen to work for polynomial approximation in the Drury-Arveson space (and other CNP spaces where the monomials form an orthogonal basis, such as the Dirichlet space). Indeed, one need only observe that, in such a space, if $a$ is any multiplier and $p$ is a polynomial, then $M_a^*p$ is a polynomial of equal or lower degree, so that one may run the same argument taking $g_1, g_2$ and $g$ to be appropriately chosen polynomials. In particular, while the original version of this argument for $\mathcal H(b)$ spaces in the Hardy space $H^2$ relied on constructions with outer functions, we have here appealed to the weak-* approximation lemma instead, which allows the argument to run in this greater generality.

\section{Questions and further remarks}\label{sec:questions}

The fundamental question concerning densely defined multipliers of CNP spaces is the following, which remains open:

{\bf Question 1:} If $\mathcal H$ is a CNP space and $h$ is a densely defined multiplier of $\mathcal H$, must $h$ belong to the Smirnov class $N^+(\cH)$?

We recall that the {\em Smirnov class} associated to $\mathcal H$ is the set $N^+(\mathcal H)$ of functions expressible in the form
\[
  N^+(\mathcal H) :=\{ \frac{\varphi}{\psi} :\varphi, \psi\in \mr{Mult} \, \cH, \ \psi \text{ cyclic }\}.
\]

It is evident that every $h\in N^+(\mathcal H)$ is a densely defined multiplier, since $\psi \mathcal H\subset \nbdom T$ and by definition $\psi\mathcal H$ is dense in $\mathcal H$ when $\psi$ is cyclic. Besides the classical Hardy space $H^2$ and the cases already discussed in Section~\ref{sec:affiliated}, (and some closely related spaces), there are no other CNP spaces in which we know the answer.

The essential difficulty is determining whether or not $\nbdom T$ must contain a cyclic function. In fact the next lemma shows this would settle the problem. 
\begin{lemma}Let $h$ be a densely defined multiplier of $\mathcal H$. Then $h\in N^+(\cH)$ if and only if $\nbdom T$ contains a cyclic vector. 
\end{lemma}
\begin{proof}
If $h$ is Smirnov, write $h=b/a$ where $a,b$ are multipliers and $a$ is cyclic; then $a\in \nbdom T$. Conversely, if $f$ is a cyclic vector in $\nbdom T$, then $f$ is Smirnov, so we can write $f=f_1/f_2$ where $f_1, f_2$ are multipliers, and $f_2$ is cyclic, but since $f$ is cyclic we must have $f_1$ cyclic as well. Let $g=hf$; then $g$ is in $\mathcal H$ and also Smirnov, write $g=g_1/g_2$ with $g_2$ cyclic. Then
\begin{equation}\frac{hf_1}{f_2} = \frac{g_1}{g_2}, \quad \text{hence} \quad h=\frac{f_2g_1}{g_2f_1}.
\end{equation}
Since $g_2, f_1$ are cyclic multipliers, so is their product, and we conclude $h$ is Smirnov.
\end{proof}
Notation: for a function $f\in \mathcal H$, let $[f]$ denote the closed, $\mr{Mult} \, \cH$-invariant subspace generated by $f$.
\begin{lemma}\label{lem:cancellation}If $h$ is a densely defined multiplier of a CNP space $\mathcal H$ and $a,b$ is a representing pair for $h$, then $b\in [a]$. 
\end{lemma}
\begin{proof} A function $f\in \mathcal H$ belongs to $\nbdom T$ if and only if there exists $g\in \mathcal H$ with $bf=ag$. In other words, $b$ multiplies $\nbdom T$ into $[a]$. Since $\nbdom T$ is assumed dense, there is a sequence $(f_n)$ from $\nbdom T$ converging to $1$ in the $\mathcal H$ norm, and hence $bf_n\to b$. For the associated $g_n$ we then have $ag_n=bf_n\to b$, and thus $b\in [a]$.
\end{proof}
In one variable, the converse of Lemma~\ref{lem:cancellation} holds. Indeed, if $a=\varphi a_1$ is the inner-outer factorization of $a$, then the hypothesis $b\in [a]$ implies that $\varphi$ divides $b$, say $b=\varphi b_1$ with $b_1\in H^\infty$, so that we can cancel the factors $\varphi$ to write $b/a=b_1/a_1$, which is Smirnov and hence the multiplier is densely defined. The converse fails in several variables, however, as the following example shows.

\begin{prop} There exist functions $f_1, f_2$ in the two-variable Drury-Arveson space $H^2_2$ such that $f_1\in [f_2]$ and $h=f_1/f_2$ is holomorphic in $\mathbb B^2$, but $\nbdom h$ is not dense in $H^2_2$. 
\end{prop}
\begin{proof}
We begin by noting that, from the definition of the Drury-Arveson norm, a function of the form
\[
\sum_{n=0}^\infty c_n z_2z_1^n
\]
belongs to $H^2_2$ if and only if
\[
\sum_{n=0}^\infty \frac{1}{n+1}|c_n|^2<\infty.
\]
Thus, for any one-variable function $g(z)$, the function $z_2g(z_1)$ belongs to $H^2_2$ if and only if $g$ belongs to the Bergman space $A^2(\mathbb D)$. On the other hand, a function $g(z_1)$ of the variable $z_1$ alone, belongs to $H^2_2$ if and only if $g$ belongs to the Hardy space $H^2(\mathbb D)$. 

To construct our example, we choose a function $g$ in the Bergman space $A^2(\mathbb D)$ whose zero set is not a Blaschke sequence. Put $f_1(z_1, z_2)=z_2g(z_1)$ and $f_2(z_1,z_2)=z_2$. Letting $p_n$ be the sequence of Taylor polynomials of $g(z)$, the functions $z_2p_n(z_1)$ converge in the $H^2_2$ norm to $f_1$, so $f_1\in [f_2]$, and clearly $f_1/f_2=g(z_1)$ is holomorphic. To see that multiplication by $g(z_1)$ does not have dense domain, we first observe that if $f$ is any $H^2_2$ function, then $f(z_1,0)$ belongs to the one-variable Hardy space $H^2(\mathbb D)$. Let $f\in H^2_2$ and suppose $g(z_1)f(z_1,z_2)$ belongs to $H^2_2$ as well. We can write $f$ in the form 
\[
f(z_1, z_2) = f(z_1,0)+z_2 F(z_1, z_2)
\]
where $F$ is holomorphic in $\mathbb B^2$. Then
\[
g(z_1)f(z_1, z_2) = g(z_1)f(z_1,0) + z_2g(z_1)F(z_1,z_2). 
\]
But if this function lies in $H^2_2$, the first summand above, a function of the variable $z_1$ only, must belong to the one variable Hardy space. But since the zero set of $g$ is not a Blaschke sequence, this is possible only if $f(z_1,0)$ vanishes identically. This implies that every function $f$ in $\nbdom h$ has the form $f=z_2F$, so this domain is not dense. 
\end{proof}

The example is easily modified to produce functions $f_1, f_2$ that are multipliers---since these functions belong to $H^2_2$, they belong to the Smirnov class and can be written as $f_j=b_j/a_j$. From this it easily follows that if $f_1\in [f_2]$, then $b_1\in [b_2]$. 

As a final remark on this question, let us note that while the ideal
\[
\mathcal A:=\{ A\Phi: \Phi\in \mr{Mult} \mathcal H, \mathcal H\otimes \mathcal E \} \subset \mr{Mult} \, \cH
\]
is weak-* dense in $\mr{Mult} \, \cH$ (since $A$ is row outer), this all by itself will not guarantee that $\mathcal A$ contains a cyclic vector. For example, in the case of $H^\infty$ of the unit disk, if we fix an infinite Blaschke sequence $Z=\{z_n\}$ and let $\mathcal I\subset H^\infty$ be the set of all $H^\infty$ functions which vanish on a cofinite subset of $Z$, then $\mathcal I$ is weak-* dense ideal in $H^\infty$, but contains no outer functions.

Finally, we recall that in the classical theory of deBranges-Rovnyak spaces in the disk, the spaces $\mathcal H(b)$ which are domains of some $T^*$ are invariant for multiplication by $z$ (since they are always invariant for the backward shift, they are thus ``doubly shift-invariant,'' per the title of \cite{sarason-1986b}). The natural analogue of this statement for the Drury-Arveson spaces is open:

{\bf Question 2:}   If $h$ is a densely defined multiplier of the Drury-Arveson space $H^2_d$, is $\nbdom T ^*$ invariant under multiplication by $z_1, \dots, z_d$?

\begin{remark}
  The full Fock space over $\C ^d$ can be defined as a Hilbert space of square-summable power series in several non-commuting variables, and is hence a natural non-commutative analogue of the classical Hardy space and the multi-variable Drury-Arveson space.  We note that in this non-commutative setting, the appropriate analogs of both Question 1 and Question 2 have positive answers \cite[Corollary 4.26, Corollary 4.27]{JM-freeSmirnov}.
\end{remark}

\bibliographystyle{alpha}
\bibliography{unbounded-CNP}

\end{document}